\documentclass[runningheads]{llncs}
\usepackage{graphicx}
\usepackage{amsmath,amsfonts}
\usepackage{multirow}
\usepackage{authblk}
\usepackage{cite}
\usepackage{url}
\usepackage{pgf}
\usepackage{color}
\usepackage{todonotes}
\usepackage{booktabs} 
\usepackage{algorithm}
\usepackage[noend]{algorithmic}
\usepackage{pgfplots}
\usepackage{caption}
\usepackage{soul}
\usepackage{bbm}
\usepackage{enumerate}

\begin{document}
\title{Bilevel Optimization for \\ On-Demand Multimodal Transit Systems}
%
%
\author{Beste Basciftci \and 
Pascal Van Hentenryck}
\authorrunning{Basciftci and Van Hentenryck}
%
\institute{Georgia Institute of Technology, Atlanta, GA 30332 \\
\email{beste.basciftci@gatech.edu, pascal.vanhentenryck@isye.gatech.edu}}
\maketitle              
\begin{abstract}
  This study explores the design of an On-Demand Multimodal Transit
  System (ODMTS) that includes segmented mode switching models that decide
  whether potential riders adopt the new ODMTS or stay
  with their personal vehicles. It is motivated by the desire of transit
  agencies to design their network by taking into account both
  existing and latent demand, as quality of service improves. The
  paper presents a bilevel optimization where the leader problem
  designs the network and each rider has a follower problem to decide
  her best route through the ODMTS. The bilevel model is solved by a
  decomposition algorithm that combines traditional Benders cuts with
  combinatorial cuts to ensure the consistency of mode choices by the
  leader and follower problems. The approach is evaluated on a case
  study using historical data from Ann Arbor, Michigan, and a user
  choice model based on the income levels of the potential transit
  riders.  
\end{abstract}
\section{Introduction}

On-Demand Multimodal Transit Systems (ODMTS) \cite{Arthur2019,ISE2019}
combines on-demand shuttles with a bus or rail network. The on-demand
shuttles serve local demand and act as feeders to and from the
bus/rail network, while the bus/rail network provides high-frequency
transportation between hubs. By using on-demand shuttles to pick up
riders at their origins and drop them off at their destinations, ODMTS
addresses the first/last mile problem that plagues most of the transit
systems. Moreover, ODMTS addresses congestion and economy of scale by
providing high-frequency along congested corridors. They have been shown
to bring substantial convenience and cost benefits in simulation and
pilot studies in the city of Canberra, Australia and the city of Ann
Arbor, Michigan. 

The design of an ODMTS is a variant of the hub-arc location problem
\cite{Campbell2005a, Campbell2005b}: It uses an optimization model
that decides which bus/rail lines to open in order to maximize
convenience (e.g., minimize transit time) and minimize costs
\cite{Arthur2019}. This optimization model uses, as input, the current
demand, i.e., the set of origin-destination pairs over time in the
existing transit system. Transit agencies however are worried about
latent demand: As the convenience of the transit system improves, more
riders may decide to switch modes and adopt the ODMTS instead of
traveling with their personal vehicles. By ignoring the latent demand,
the ODMTS may be designed suboptimally, providing a lower convenience
or higher costs. This concern was raised in \cite{Campbell2019} who
articulated the potential of leveraging data analytics within the
planning process and proposing transit systems that encourage riders
to switch transportation modes.

This paper aims at remedying this limitation and explores the design
of ODMTS with both existing and latent demands. It considers a pool of
potential riders, each of whom is associated with a personalized mode
choice model that decides whether a rider will switch mode for a given
ODMTS. Such a choice model can be obtained through stated and revealed
preferences, using surveys and/or machine learning \cite{Zhao2019}.
The main innovation of this paper is to show how to integrate such
mode choice models into the design of ODMTS, capturing the latent
demand and human behavior inside the optimization model. More precisely,
the contributions of the paper can be summarized as follows:

\begin{enumerate}
\item The paper proposes a novel bilevel optimization approach to model the ODMTS problem with latent demand in order to obtain the most cost-efficient and convenient route for each trip.
\item The bilevel optimization model includes a personalized mode choice for each rider to determine mode switching or latent demand.
\item The bilevel optimization model is solved through a decomposition algorithm that combines both traditional and combinatorial Benders cuts. 
\item The paper demonstrates the benefits and practicability of the approach on a case study using historical data over Ann Arbor, Michigan.
\end{enumerate}

The remainder of the paper is organized as follows. Section
\ref{Sectionliterature} reviews the relevant literature. Section
\ref{section:problem} specifies the ODMTS design problem. Section
\ref{SectionModel} proposes a bilevel optimization approach for the
design of ODMTS with latent demand, and Section
\ref{SectionSolutionAlg} develops the novel decomposition
methodology. The case study is presented in in Section
\ref{SectionComputations} and Section \ref{SectionConclusion}
concludes the paper with final remarks.

\section{Related Literature}
\label{Sectionliterature}

Hub location problems are an important area of research in transit
network design (see \cite{Farahani2013} for a recent review). More
specifically, the transit network design problem considering hubs can
be considered as a variant of the hub-arc location problem
\cite{Campbell2005a, Campbell2005b}, which focuses on determining the
set of arcs to open between hubs, and optimizing the flow with minimum
cost. Mah\'eo, Kilby, and Van Hentenryck \cite{Arthur2019} extended this
problem to the ODMTS setting by introducing on-demand shuttles and
removing the restriction that each route needs to contain an arc
involving a hub. Furthermore, instead of the restriction of hubs being
interconnected in the network design, they consider a weak
connectivity within system by ensuring the sum of incoming and
outgoing arcs to be equal to each other for each hub. Although these
studies provide efficient solutions for a given demand, they neglect
the effect of the latent demand which can change the design of the
transit systems.

Bilevel optimization is an important area of mathematical programming,
which mainly considers a leader problem, and a follower problem that
optimizes its decisions under the leader problem's output. Due to this hierarchical decision making structure,
this area attracted attention in different urban transit network
design applications \cite{Leblanc1986, Farahani2013_Review}  such as discrete network design problems \cite{Fontaine2014} by improving a network via adding lines or increasing their capacities, and bus lane design problems \cite{Yu2015} under traffic equilibrium. Studies \cite{Colson2005, Sinha2018} provide an overview of various solution methodologies to address these problems including reformulations based on Karush-Kuhn-Tucker (KKT) conditions, descent methods and heuristics. User preferences and the corresponding latent demand constitute important factors impacting the network design. Because of the computational complexity involved with solving bilevel problems, it is preferred to model rider preferences within a single level optimization problem \cite{Laporte2011}. To this end, our approach provides a novel bilevel optimization framework for solving the ODMTS by integrating user choices, and developing an exact decomposition algorithm as its solution procedure. 

\section{Problem Statement}
\label{section:problem}

This section defines the problem statement and stays as close as
possible to the original setting of the ODMTS design
\cite{Arthur2019}. In particular, the input consists of a set of
(potentially virtual) bus stops $N$, a set of potential hubs $H
\subseteq N$, and a set of trips $T$. Each trip $r \in T$ is
associated with an origin stop $or^r \in N$, a destination stop $de^r
\in N$, and a number of passengers taking that trip $p^r \in
\mathbb{Z}_{+}$. This paper often abuses terminology and uses trips
and riders interchangeably, although a trip may contain several
riders. The distance and time between each node pair $i, j \in N$ is
given by parameters $d_{ij}$ and $t_{ij}$, respectively. These
parameters can be asymmetric and are assumed to satisfy the triangular
inequality. The network design optimizes a convex combination of
convenience (mostly travel time) and cost, using parameter $\theta \in
[0,1]$: In other words, convenience is multiplied by $\theta$ and cost
by $1 - \theta$. The investment cost of opening a leg between the hubs
$h, l \in H$ is given by as $\beta_{hl} = (1 - \theta) \ b \ n
\ d_{hl}$, where $b$ is the cost of using a bus per mile and $n$ is
the number of buses during the planning period. For each trip $r \in
T$, the weighted cost and convenience of using a bus between the hubs
$h, l \in H$ is given by $\tau^r_{hl} = \theta (t_{hl} + S)$, where
$S$ is the average waiting time of a bus (the bus cost is covered by
the investment).

This paper adopts a pricing model where the ODMTS subsidizes part, but
not all, of the shuttle costs. More precisely, for simplicity in the
notations, the paper assumes that the transit price is half of the
shuttle cost of a trip.\footnote{The results in this paper generalize
  to other subsidies and pricing models, and they will be discussed in
  the extended version of the paper.} With this pricing model, the
weighted cost and convenience for an on-demand shuttle between $i$ and
$j$ for the ODMTS and riders is given by $(1 - \theta) \ \frac{g}{2}
\ d_{ij} + \theta t_{ij}$, where $g$ is the cost of using a shuttle
per mile. Moreover, the shuttles act as feeders to bus system or
serve the local demand. As a result, their operations are restricted
to serve requests in a certain distance. This is captured by a
threshold value of $\Delta$ miles that characterizes the trips that
shuttles can serve. As a result, it is suitable to define the
weighted cost and convenience of an on-demand shuttle between the
stops $i,j \in N$ as follows:
\begin{equation*}
\gamma^r_{ij} := \begin{cases}
(1 - \theta) \frac{g}{2} d_{ij} + \theta t_{ij} & \text{if } d_{ij} \leq \Delta \\
M & \text{if } d_{ij} > \Delta. \\
\end{cases}
\end{equation*}
where $M$ is a big-M parameter. 
 
To capture latent ridership, this paper assumes that a subset of trips
$T' \subseteq T$ currently travel with their personal vehicles, while the
trips in $T \setminus T'$ already use the transit system. The goal of
the paper is to capture, in the design of the ODMTS, the fact that
some riders may switch mode and use the ODMTS instead of their own
cars as the transit system has a better cost/convenience
trade-off. Each rider $r \in T'$ has a choice model ${\cal C}^r$ that
  determines, given the cost/convenience of the novel ODMTS, whether
  $r$ will switch to the transit system. For instance, the cost model
  could be
\[
{\cal C}^r(d^r) \equiv \mathbbm{1}(d^r \leq \alpha^r \ d^r_{car})
\]
where $d^r_{car}$ represents the weighted cost and convenience of
using a car for rider $r$, $d^r$ represents the weighted cost and
convenience of using the ODMTS in some configuration, and $\alpha^r
\in \mathbb{R}_{+}$. In other words, rider $r$ would switch to transit
if its convenience and cost is not more than $\alpha^r$ times the cost
and convenience of traveling with her personal vehicle. The choice
model could of course be more complex and include the number of
transfers and other features. It can be learned using multimodal logic
models or machine learning \cite{Zhao2019}.

\section{Model Formulation}
\label{SectionModel}

This section proposes an optimization model for the design of an ODMTS
following the specification from Section \ref{section:problem}. In the
model, binary variable $z_{hl}$ is 1 if there is a bus connection from
hub $h$ to $l$. Furthermore, for each trip $r$, binary variables
$x_{hl}^r$ and $y_{ij}^r$ represent whether rider $r$ uses a bus leg
between hubs $h,l \in H$ and a shuttle leg between stops $i$ and $j$
respectively. Binary variable $\delta^r$ for $r \in T'$ is 1 if rider
$r$ switches to the ODMTS. The bilevel optimization model for the
ODMTS design can then be specified as follows:

\allowdisplaybreaks
\begin{subequations} \label{eq:upperLevelProblemUpdated2}
\begin{alignat}{1}
\min \quad & \sum_{h,l \in H} \beta_{hl} z_{hl} + \sum_{r \in T \setminus T'} p^r d^r + \sum_{r \in T'} p^r \delta^r d^r \label{eq:upperLevelObj} \\
\text{s.t.} \quad & \sum_{l \in H} z_{hl} = \sum_{l \in H} z_{lh} \quad \forall h \in H \label{eq:upperLevelConstr1} \\ 
& \delta^r = {\cal C}^r(d^r) \quad \forall r \in T' \label{eq:userChoiceModel} \\
& z_{hl} \in \{0,1\} \quad \forall h,l \in H  \label{eq:binaryConstraint} \\
& \delta^r \in \{0,1\} \quad \forall r \in T' \label{eq:continuousConstraint} 
\end{alignat}
\end{subequations}
where $d^r$ is the cost and convenience of trip $r$, i.e., 
\begin{subequations}  \label{eq:lowerLevelProblem}
\begin{alignat}{1}
d^r = \min \quad & \sum_{h,l \in H}  \tau_{hl}^r x_{hl}^r + \sum_{i,j \in N} \gamma_{ij}^r y_{ij}^r  \label{eq:lowerLevelObj} \\
\text{s.t.} \quad & \sum_{\substack{h \in H \\ \text{if } i \in H}} (x_{ih}^r - x_{hi}^r) + \sum_{i,j \in N}  (y_{ij}^r - y_{ji}^r) = \begin{cases}
1 &, \text{if  } i = or^r \\
-1 &, \text{if  } i = de^r \\
0 &, \text{otherwise}
\end{cases}
\quad \forall i \in N \label{eq:minFlowConstraint} \\
& x_{hl}^r \leq z_{hl} \quad \forall h,l \in H \label{eq:openFacilityOnlyAvailable} \\
& x_{hl}^r \in \{0,1\} \quad \forall h,l \in H, y_{ij}^r \in \{0,1\} \quad \forall i,j \in N. \label{eq:integralityFlowConstr}
\end{alignat}
\end{subequations}

The resulting formulation is a bilevel optimization where the leader
problem (equations \eqref{eq:upperLevelObj}--
\eqref{eq:continuousConstraint}) selects the network design and the
follower problem (equations
\eqref{eq:lowerLevelObj}--\eqref{eq:integralityFlowConstr}) computes
the weighted cost and convenience for each rider $r \in T$ in the
proposed ODMTS.

The objective of the leader problem \eqref{eq:upperLevelObj} minimizes
the investment cost of opening legs between hubs and the weighted cost
and convenience of the routes in the ODMTS for those
riders. Constraint \eqref{eq:upperLevelConstr1} ensures weak
connectivity between the hubs and constraint
\eqref{eq:userChoiceModel} represents the rider choice, i.e., whether
rider $r \in T'$ switches to the ODMTS.

The follower problem of a given trip minimizes the cost and
convenience of its route between its origin and destination, under a
given transit network design between the hubs (objective function
\eqref{eq:lowerLevelObj}). Constraint \eqref{eq:minFlowConstraint}
ensures flow conservation for the bus and shuttle legs.  Constraint
\eqref{eq:openFacilityOnlyAvailable} guarantees that only open legs
are considered by each trip. The follower problem has a totally
unimodular constraint matrix, once the leader problem determines the
transit network design decisions $z$. In this case, integrality
restrictions \eqref{eq:integralityFlowConstr} can be relaxed, and the
problem can be solved as a linear program.

As specified, the follower problem takes into account all of the arcs
between each node pair $i,j \in N$ for possible rides with on-demand
shuttles. However, due to the triangular inequality, it is sufficient
to consider a subset of the arcs for the on-demand shuttles of each
trip. More precisely, the optimization only needs to consider arcs i)
from origin to hubs, ii) from hubs to destination, and iii) from
origin to destination. This subset of necessary arcs for trip $r$ is
denoted by $A^r$. Consequently, the model only needs the following
decision variables for describing the on-demand shuttles used in trip
$r$:
\begin{align*}
y^r_{or^r h}, y^r_{h de^r} & \in \{0,1\} \quad \forall h \in H \\
y^r_{or^r de^r} & \in \{0,1\}.
\end{align*}
This preprocessing step significantly reduces the size of the follower
problem and provides a significant computational benefit.

\section{Solution Methodology}
\label{SectionSolutionAlg}

This section presents a decomposition approach to solve the bilevel
problem \eqref{eq:upperLevelProblemUpdated2}. The decomposition
combines traditional Benders optimality cuts with combinatorial
Benders cuts to capture the rider choices. The Benders master problem
is associated with the leader problem and considers the
complicating variables $(z_{hl}, \delta^r, d^r)$ and the subproblems
are associated with the follower problems.  The master problem relaxes
the user choice constraint \eqref{eq:userChoiceModel}.  The duals of
the subproblems generate Benders optimality cuts for the master
problem. Moreover, combinatorial Benders cuts are used to ensure that
the rider mode choices in the master problem are correctly captured by
the master problem. The overall decomposition approach iterates
between solving the master problem and guessing $(\bar{z}_{hl},
\bar{\delta}^r, \bar{d}^r)$ and solving the subproblems to obtain the
correct value $d^r$ from which the switching decision ${\cal
  C}^r(d^r)$ can be derived. The overall process terminates when 
the lower bound obtained in the master problem and 
upper bound computed through the feasible solutions converge.

Section \ref{masterProblemSection} presents the master problem and
Section \ref{sectionSubproblem} discusses the subproblem along with
some preprocessing steps. Section \ref{cutGenerationSection}
introduces the cut generation procedure and proposes stronger cuts
under some natural monotonicity
assumptions. Section~\ref{decompositionAlgorithmSection} specifies the
proposed decomposition algorithm and proves its finite
convergence. Finally, Section~\ref{paretoCutsSection} improves the
decomposition approach with Pareto-optimal cuts.

\subsection{Relaxed Master Problem}
\label{masterProblemSection}

The initial master problem \eqref{eq:relaxedMasterProblemInitial} is a
relaxation of the bilevel
problem~\eqref{eq:upperLevelProblemUpdated2}, i.e.,

\begin{subequations} \label{eq:relaxedMasterProblemInitial}
\begin{alignat}{1}
\min \quad & \sum_{h,l \in H} \beta_{hl} z_{hl} + \sum_{r \in T \setminus T'} p^r d^r + \sum_{r \in T'} p^r \delta^r d^r \label{eq:masterObjectiveFunction} \\
\text{s.t.} \quad & \eqref{eq:upperLevelConstr1}, \eqref{eq:binaryConstraint}, \eqref{eq:continuousConstraint}. \notag 
\end{alignat}
\end{subequations}

\noindent
Each iteration first solves the relaxed master problem
\eqref{eq:relaxedMasterProblemInitial}, before identifying
combinatorial and Benders cuts to add to the master problem. These cuts
depend on the proposed transit network design and rider choices as
discussed in Sections \ref{sectionSubproblem} and
\ref{cutGenerationSection}. The objective function
\eqref{eq:masterObjectiveFunction} involves nonlinear terms and needs
to be linearized. Since the mode choice is binary, the nonlinear
terms can be linearized easily by defining $\nu^r = \delta^r d^r$ and
adding the following constraints to the master problem for each trip
$r \in T'$:
\begin{subequations} 
\begin{alignat}{1}
\nu^r & \leq \bar{M}^r \delta^r \\
\nu^r & \leq d^r \\
\nu^r & \geq d^r - \bar{M}^r (1 - \delta^r) \\
\nu^r & \geq 0,
\end{alignat}
\end{subequations}
where the constant $\bar{M}^r$ is an upper bound value on the
objective function value of the lower level problem of trip $r$.


\subsection{Subproblem for each Trip}
\label{sectionSubproblem}

The subproblems for the decomposition algorithm are the duals of the
follower problems \eqref{eq:lowerLevelProblem}. Since the follower
problems have a totally unimodular constraint matrix for a given
binary $\bar{z}$ vector, the integrality condition for variable
$x_{ij}^r$ can be relaxed into by $x^r_{ij} \geq 0$ and the bounds
$x^r_{ij} \leq 1$ can be discarded since it is redundant due to
constraint \eqref{eq:openFacilityOnlyAvailable}. Then, the dual of the
subproblem for each route $r \in T$ can then be specified by
introducing the dual variables $u^r_{i}$ and $v_{hl}^r$:

\begin{subequations} \label{eq:dualLowerLevel}
\begin{alignat}{1}
\max \quad & (u_{or^r}^r - u_{de^r}^r) - \sum_{h,l \in H} \bar{z}_{hl} v^r_{hl} \\
\text{s.t.} \quad & u_h^r - u_l^r - v_{hl}^r \leq \tau_{hl}^r \quad \forall h,l \in H \label{eq:dualLowerLevelConstr1} \\
& u_i^r - u_j^r \leq \gamma_{ij}^r \quad \forall i,j \in A^r \label{eq:dualLowerLevelConstr2} \\
& u_i^r \geq 0 \quad \forall i \in N, v_{hl}^r \geq 0 \quad \forall h,l \in H. \label{eq:dualLowerLevelConstr3}
\end{alignat}
\end{subequations}
Problem \eqref{eq:lowerLevelProblem} is trivially feasible by using the
direct trip between origin and destination (which may have a high
cost) and hence the dual problem \eqref{eq:dualLowerLevel} bounded.
The optimal objective value of subproblem \eqref{eq:lowerLevelProblem}
under solution $\{\bar{z}_{hl}\}_{h,l \in H}$ is denoted by
$SP^r(\bar{z})$. In the following section, this value is utilized to
evaluate the rider's mode choice and possibly to generate combinatorial
cuts.

\subsection{The Cut Generation Procedure}
\label{cutGenerationSection}

The cut generation procedure receives a feasible solution
$(\{\bar{z}_{hl}\}_{h,l \in H}$, $\{\bar{\delta}^r\}_{r \in T'}$,
$\{\bar{d}^r\}_{r \in T})$ to the relaxed master problem.  It solves the dual
subproblem \eqref{eq:dualLowerLevel} for each trip $r \in T$ under the
network design $\bar{z}$. For any trip $r \in T'$, the cut generation
procedure then analyzes the feasibility and optimality of the solution
of the relaxed master problem, depending on the value of
$SP^{r}(\bar{z})$. The cut generation first needs to enforce the
consistency of the choice model. 

\begin{definition}[Choice Consistency]
For a given trip $r$, the solution values $\{\bar{z}_{hl}\}_{h,l \in H}$ and $\bar{\delta}^{r}$
are \textit{consistent} with $SP^{r}(\bar{z})$ if
\[
\bar{\delta}^{r} = {\cal C}^r(SP^{r}(\bar{z})).
\]
\end{definition}

\noindent
As a result, it is useful to distinguish the following cases in the cut generation process:
\begin{enumerate}
\item Solution values $\{\bar{z}_{hl}\}_{h,l \in H}$ and $\bar{\delta}^{r}$ are \textit{inconsistent} with $SP^{r}(\bar{z})$
\begin{enumerate}
\item $\bar{\delta}^{r} = 1$ and ${\cal C}^r(SP^{r}(\bar{z})) = 0$;
\item $\bar{\delta}^{r} = 0$ and ${\cal C}^r(SP^{r}(\bar{z})) = 1$.
\end{enumerate}
\item Solution values $\{\bar{z}_{hl}\}_{h,l \in H}$ and $\bar{\delta}^{r}$ are \textit{consistent} with $SP^{r}(\bar{z})$.
\end{enumerate}

\noindent
The first inconsistency (case 1(a)) can be removed by using the cut
\begin{equation} \label{eq:case2Cut}
\sum_{(h,l):\bar{z}_{hl}=0} z_{hl} + \sum_{(h,l):\bar{z}_{hl}=1} (1-z_{hl}) \geq \delta^{r}
\end{equation}
\begin{proposition} \label{nogoodCut1Prop}
Constraint \eqref{eq:case2Cut} removes inconsistency 1(a).
\end{proposition}

\noindent
The second inconsistency (case 1(b)) can be removed by using the cut
\begin{equation} \label{eq:case3Cut}
\sum_{(h,l):\bar{z}_{hl}=0} z_{hl} + \sum_{(h,l):\bar{z}_{hl}=1} (1-z_{hl}) + \delta^{r} \geq 1
\end{equation}
\begin{proposition} \label{nogoodCut2Prop}
Constraint \eqref{eq:case3Cut} removes inconsistency 1(b).
\end{proposition}

\noindent
Combinatorial cuts \eqref{eq:case2Cut} and \eqref{eq:case3Cut} ensure
the consistency between the rider choice model and the transit network
design $\bar{z}$. These cuts can be strenghtened under a monotonicity
property.


\begin{definition}[Anti-Monotone Mode Choice] A choice function ${\cal C}$ is anti-monotone if
$
  d_1 \leq d_2 \Rightarrow {\cal C}(d_1) \geq {\cal C}(d_2).
$
\end{definition}

\begin{proposition} \label{monotoneTheorem}
Let $r \in T$. If $\bar{z}_1 \leq \bar{z}_2$, then $SP^{r}(\bar{z}_1) \geq SP^{r}(\bar{z}_2)$.
\end{proposition}
\begin{proof}
If $\bar{z}_1 \leq \bar{z}_2$, more arcs are available in the network
defined by $\bar{z}_2$ than in the network defined by
$\bar{z}_1$. Therefore, the length of the optimum shortest path for
trip $r$ under $\bar{z}_1$ is greater than or equal to that of
$\bar{z}_2$. \qed
\end{proof}

\noindent
The following proposition follows directly from Proposition \ref{monotoneTheorem}.

\begin{proposition} \label{monotoneProposition}
Let $r \in T$ and ${\cal C}^r$ be an anti-monotone choice function. If $\bar{z}_1 \leq \bar{z}_2$, then ${\cal C}^r(SP^{r}(\bar{z}_1)) \leq {\cal C}^r(SP^{r}(\bar{z}_2))$.
\end{proposition}

\noindent
When the choice function is anti-monotone, stronger cuts can be derived. 


\begin{proposition} \label{nogoodCut1PropStronger}
Consider an anti-monotone choice function. Then constraint \eqref{eq:case2Cut} for case 1(a) can be strengthened into constraint 
\begin{equation} \label{eq:case2CutStronger}
\sum_{(h,l):\bar{z}_{hl}=0} z_{hl} \geq \delta^{r}
\end{equation}
\end{proposition}

\begin{proof}
  Consider case 1(a) and network design $\bar{z}$. Let $\tilde{z}$ be
  a network design obtained by removing some arcs from $\bar{z}$. By
  Proposition \ref{monotoneTheorem}, $SP^{r}(\tilde{z}) \geq
  SP^{r}(\bar{z})$ for any trip $r$. Hence, by Proposition
  \ref{monotoneProposition}, ${\cal C}^r(SP^{r}(\tilde{z})) \leq {\cal
    C}^r(SP^{r}(\bar{z}))$. Therefore, the right term of cut
  \eqref{eq:case2Cut} does not remove the inconsistency and the result
  follows. \qed
\end{proof}  

\begin{proposition}
Consider an anti-monotone choice function. Then constraint \eqref{eq:case3Cut} for case 1(b) can be strengthened into constraint
\begin{equation} \label{eq:case3CutStronger}
\sum_{(h,l):\bar{z}_{hl}=1} (1-z_{hl}) + \delta^{r} \geq 1
\end{equation}
\end{proposition}
\begin{proof}
  Consider case 1(b) and network design $\bar{z}$. Let $\tilde{z}$ be
  a network design obtained by adding some arcs to $\bar{z}$. By
  Proposition \ref{monotoneTheorem}, $SP^{r}(\tilde{z}) \leq
  SP^{r}(\bar{z})$ for any trip $r$. Hence, by Proposition
  \ref{monotoneProposition}, ${\cal C}^r(SP^{r}(\tilde{z})) \geq {\cal
    C}^r(SP^{r}(\bar{z}))$. Thus, the left term of cut
  \eqref{eq:case3Cut} does not remove the inconsistency and the result
  follows. \qed
\end{proof}


\noindent
Since the dual subproblem \eqref{eq:dualLowerLevel} is bounded, it is
also possible to add an optimality cut to the master problem in both
cases of 1 and 2 using the weighted cost and convenience of each
obtained route. This cut is the standard Benders optimality cut and it
uses the vertex $(\bar{u}^r, \bar{v}^r)$ obtained when solving the
dual subproblem as follows:
\begin{equation} \label{eq:optimalityCutLowerLevel} 
d^{r} \geq (\bar{u}^{r}_{or^{r}} - \bar{u}^{r}_{de^{r}}) - \sum_{h,l \in H} z_{hl} \bar{v}_{hl}^{r}.
\end{equation}

\noindent
It is also possible to obtain an upper bound from the solutions to the
subproblems. Indeed, the rider choices can be derived from the
solutions of the subproblems and used instead of the corresponding
master variables for the mode choices.

The experimental results use the choice function ${\cal C}^r(d^r)
\equiv \mathbbm{1}(d^r \leq \alpha^r \ d^r_{car})$: A rider $r$ chooses the ODMTS
if her weighted cost and convenience is not greater than $\alpha^r$
times the weighted cost and convenience $d^r_{car}$ of using her
personal car. This choice function is anti-monotone.
\begin{proposition}
The choice function ${\cal C}^r(d^r) \equiv \mathbbm{1} (d^r \leq \alpha^r \ d^r_{car})$ is anti-monotone. 
\end{proposition}
\begin{proof}
  By definition, $d^r$ decreases when adding arcs to a network and $d_1^r \leq d_2^r$ implies
  ${\cal C}^r(d^r_1) \geq {\cal C}^r(d^r_2)$. \qed
\end{proof}

\subsection{Decomposition Algorithm} 
\label{decompositionAlgorithmSection}

The decomposition is summarized in Algorithm
\ref{alg:bilevelNetworkDecomposition}. It uses a lower and an upper
bound to the bilevel problem \eqref{eq:upperLevelProblemUpdated2} to
derive a stopping condition. The master problem provides a lower bound
and, as mentioned earlier, an upper bound can be derived for each
network design by solving the subproblems and obtaining the mode
choices for the trips.

\begin{algorithm}[!t]
\caption{Decomposition Algorithm}
\label{alg:bilevelNetworkDecomposition}
\begin{algorithmic}[1]
\STATE Set $LB = -\infty$,  $UB = \infty$, $z^*=\emptyset$. 
\WHILE{$UB > LB + \epsilon$}
\STATE Solve the relaxed master problem \eqref{eq:relaxedMasterProblemInitial} and obtain the solution ($\{\bar{z}_{hl}\}_{h,l \in H}$, $\{\bar{\delta}^r\}_{r \in T'}$, $\{\bar{d}^r\}_{r \in T})$.  
\STATE Update $LB$.
\FORALL{$r \in T$}
\STATE Solve the subproblem \eqref{eq:dualLowerLevel} under $\bar{z}$, and obtain $SP^{r}(\bar{z})$. 
\STATE Add optimality cut in the form \eqref{eq:optimalityCutLowerLevel} to the relaxed master problem \eqref{eq:relaxedMasterProblemInitial}. 
\ENDFOR
\FORALL{$r \in T'$}
\IF{$\{\bar{z}_{hl}\}_{h,l \in H}$ and $\bar{\delta}^{r}$ are  inconsistent with $SP^{r}(\bar{z})$}
\STATE Add cuts in the form \eqref{eq:case2CutStronger} or \eqref{eq:case3CutStronger} to the relaxed master problem. \label{algorithmNogoodCutStep}
\ENDIF
\IF{${\cal C}^r(SP^{r}(\bar{z}))$ is 1} 
\STATE Set $\hat{\delta}^r = 1$.
\ELSE
\STATE Set $\hat{\delta}^r = 0$.
\ENDIF
\ENDFOR
\STATE $\widehat{UB} = \sum_{h,l \in H} \beta_{hl} \bar{z}_{hl} + \sum_{r \in T \setminus T'} p^r SP^{r}(\bar{z}) + \sum_{r \in T'} p^r \hat{\delta}^r SP^{r}(\bar{z})$.
\IF{$\widehat{UB} < UB$} 
\STATE Update $UB$ as $\widehat{UB}$, $z^* = \bar{z}$. 
\ENDIF
\ENDWHILE
\end{algorithmic}
\end{algorithm}

\begin{proposition}
Algorithm \ref{alg:bilevelNetworkDecomposition} converges in finitely many iterations. 
\end{proposition}
\begin{proof}
  The algorithm generates traditional Benders optimality cuts and, in
  addition, the consistency cuts of the form
  \eqref{eq:case2CutStronger} or \eqref{eq:case3CutStronger}. When all
  the consistency cuts are generated, the algorithm reduces to a
  standard Benders decomposition. There are only finitely many
  consistency cuts, because the decision variables $z$ and $\delta^r$
  are binary. Since each iteration adds at least one new consistency
  or Benders cut, the algorithm is guaranteed to converge in finitely
  many iterations. \qed
\end{proof}

\subsection{Pareto-Optimal Cuts}
\label{paretoCutsSection}

The decomposition algorithm can be further enhanced by utilizing
Pareto-optimal cuts \cite{Magnanti1981} while generating the
optimality cuts \eqref{eq:optimalityCutLowerLevel}. This approach aims
at accelerating the decomposition algorithm by generating
stronger cuts through alternative optimal solutions of the
subproblems. To this end, the algorithm first solves the follower problem
\eqref{eq:lowerLevelProblem} under a given network design, obtains the
optimal objective function value for the corresponding trip, and then
solve the Pareto subproblem, which is a restricted version of the dual
subproblem \eqref{eq:dualLowerLevel} under this optimal value.

Observe that, once the transit network design $\bar{z}$ is given, the
follower problem of each trip $r$ is equivalent to solving a shortest
path problem considering the union of the arcs defined by $\bar{z}$
and the arcs in the set $A^r$. Consequently, this shortest path
information can be obtained by solving a linear program and obtaining
the objective value $\sigma^r$ for trip $r$. Using this information,
the Pareto subproblem for trip $r$ is defined as follows:
\begin{subequations} \label{eq:dualLowerLevelPareto}
\begin{alignat}{1}
\max \quad & (u_{or^r}^r - u_{de^r}^r) - \sum_{h,l \in H} {z}^0_{hl} v^r_{hl} \\
\text{s.t.} \quad & u_h^r - u_l^r - v_{hl}^r \leq \tau_{hl}^r \quad \forall h,l \in H \label{eq:dualLowerLevelConstr1Pareto} \\
& u_i^r - u_j^r \leq \gamma_{ij}^r \quad \forall i,j \in A^r \label{eq:dualLowerLevelConstr2Pareto} \\
& (u_{or^r}^r - u_{de^r}^r) - \sum_{h,l \in H} \bar{z}_{hl} v^r_{hl} = \sigma^r \label{eq:dualLowerLevelConstr4Pareto} \\
& u_i^r \geq 0 \quad \forall i \in N, v_{hl}^r \geq 0 \quad \forall h,l \in H, \label{eq:dualLowerLevelConstr3Pareto}
\end{alignat}
\end{subequations}
where $z^0$ is a core point that satisfies the weak connectivity
constraint \eqref{eq:upperLevelConstr1}. To obtain an initial core
point, it suffices to select a value $\eta \in (0,1)$, and set $z_{hl}
= \eta$ for all $h,l \in H$.

\section{Computational Results}
\label{SectionComputations}

The computational study considers a data set from Ann Arbor, Michigan
with 10 hubs located around high density corridors and 1267 bus stops.
The experiments examine a set of trips from 6 pm to 10 pm on a
specific day. The studied data set involves 1503 trips with a total of
2896 users, where the origin and destination of each trip are
associated with bus stops. The costs and times between the bus stops
are asymmetric in the studied data set. The study included a
preprocessing step to ensure the triangular inequality with respect to
the cost and convenience parameters of the on-demand shuttles between
the stops.

To model rider preferences in the formulation, the computational study
used an income-based classification. This approach assumes that, as
the income level of a rider increases, she becomes more sensitive to
the quality of the ODMTS route (convenience). In particular, the study
considers three classes of riders: i) low-income, ii) middle-income
and iii) high-income, where a certain percentage of riders from each
class is assumed to use the ODMTS. The trips are then classified with
respect to their destination locations, which can be associated with
the residences of the corresponding riders. In particular, in the base
scenario, 100\% of low-income riders, 75\% of middle-income riders,
and 50\% of high-income users utilize the transit system, whereas the
remaining riders have the option to select the ODMTS or use their
personal vehicles by comparing the obtained route with their current
mode of travel.

The convenience parameter $\theta$ is set to $0.01$ for weighting cost
and convenience. The cost of an on-demand shuttle per mile is taken as
$g = \$2.86$ and the cost of a bus per mile is $b = \$7.24$. The buses
between hubs have a frequency of 15 minutes, resulting in 16 buses
during the planning horizon with length of 4 hours. As mentioned
earlier, the price of a ride in the ODMTS is half the cost of the
shuttle legs. The base case of the case study sets $\alpha^r$ to 1.25
and 1 for middle-income and high-income riders respectively. The
distance threshold for the on-demand shuttles, $\Delta$, is set to 2
or 5 miles.

\subsection{Transit Design and Mode Switching}
\label{caseStudySection}

Figure \ref{networkDesignFigure} depicts the transit network design
between hubs under the proposed approach. The bus stops associated
with the lowest income level are red dots, those of the middle-income
level are grey boxes, and those of the high-income level are green
plus symbols. In the resulting network design, almost every hub is
connected to at least another hub ensuring weak connectivity within
the transit system.

\begin{figure}[!t]
\includegraphics[width=\textwidth]{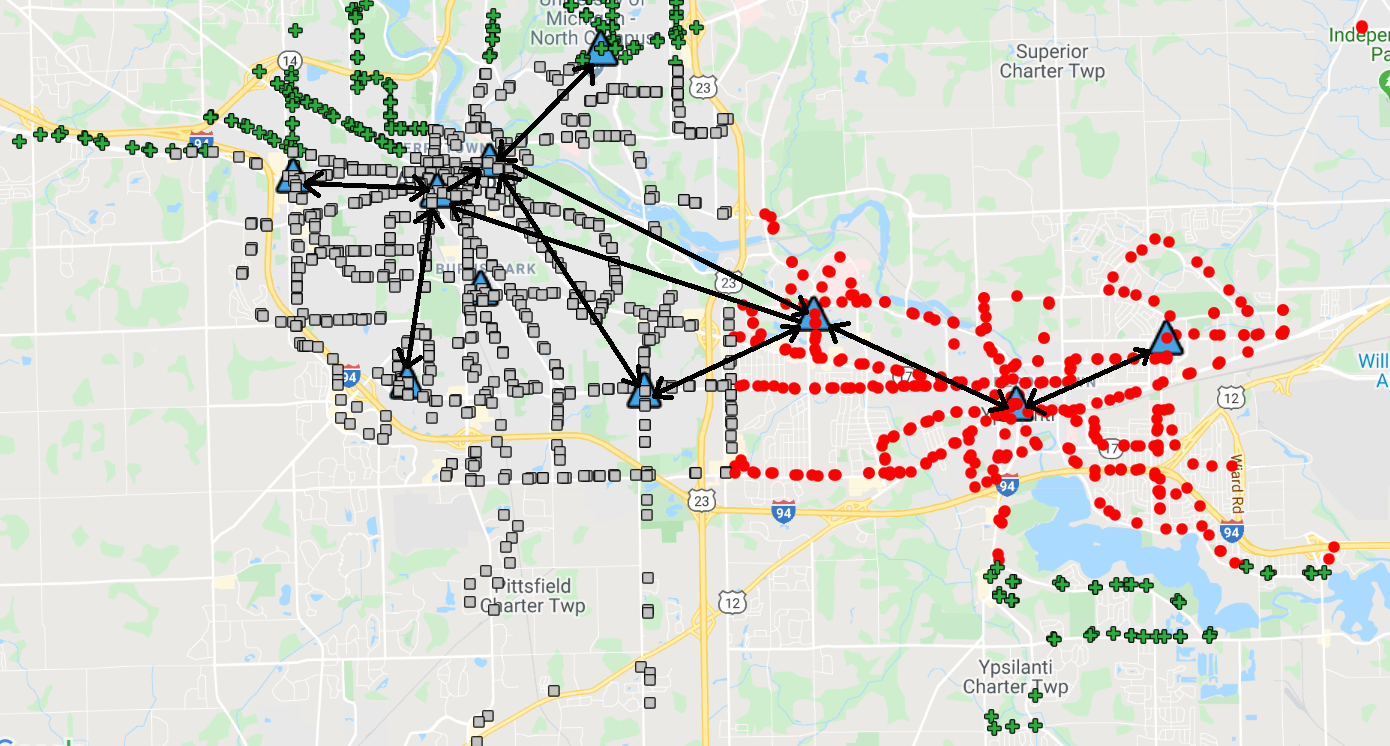}
\caption{Network Design for the ODMTS with 10 Hubs with $\Delta = 2$.} \label{networkDesignFigure}
\end{figure}

\begin{table}[!t]
\caption{Adoption Rates, Average Route Time and Average Ride Cost for the ODMTS.}
\centering
\begin{tabular}{crrrrrr}
\hline
income level & \#trips & \%adoption & \#riders & \%adoption & avg route time (s) & avg route cost (\$) \\
\hline
       low &        476 &       1.00 &        877 &       1.00 &     901.45 &       2.41 \\

    middle &        784 &       0.96 &       1615 &       0.97 &     553.43 &       2.43 \\

      high &       149 &       0.72 &        285 &       0.79 &     583.10 &       2.78 \\

\hline
\end{tabular}  
\label{caseStudyTable}
\end{table}

Table \ref{caseStudyTable} shows the rider preferences, and the
average time and cost of the obtained routes. In particular, columns
`\#trips' and `\#riders' represent the number of trips and riders of
the ODMTS. The ``\%adoption'' columns correspond to the adoption rate,
i.e., the percentage of trips or riders utilizing the ODMTS. When
computing the adoption rate, these numbers include the initial set of
riders, i.e., 100\% of low-income riders, 75\% of middle-income riders
and 50\% of high-income riders.  The cost and convenience of the ODMTS
is sufficiently attractive to exhibit significant mode switching, even
for the high-income population.  Columns for the average route time
and cost represent the averages for the obtained routes regardless of
the fact that whether riders adopts the transit system or not. The
results highlight the high adoption rates. The average route time is
the longest for the low-income riders given their long commuting
trips.  Similar results are observed for the number of transfers,
which include the transfers between on-demand shuttles and buses, and
between the buses in the hubs. Specifically, from the set of riders
choosing the transit system, 22\% of low-income riders, 8\% of
medium-income riders, and 3\% of high-income riders have at least 3
transfers. Moreover, the number of transfers decreases with increases
in income level. 

\begin{table}[!t]
\caption{Comparing the Average Cost and Time of the ODMTS trips and those Using Personal Vehicles (Cars).}
\begin{tabular}{crrrr|rrrr}
\hline
           & \multicolumn{ 4}{c}{ODMTS Trips} & \multicolumn{ 4}{c}{Car Trips} \\
\cline{2-9}
 & \multicolumn{ 2}{c}{Time} & \multicolumn{ 2}{c}{Cost} & \multicolumn{ 2}{c}{Time} & \multicolumn{ 2}{c}{Cost} \\
 \hline
income level           &      ODMTS &        Cars &      ODMTS &        Cars &      ODMTS &        Cars &      ODMTS &        Cars \\
\hline
low & 901.45 &     405.96 &       2.41 &      10.72 &         NA &            &      NA      &            \\

   medium & 528.94 &     296.95 &       2.38 &       7.14 &    1489.80 &     585.03 &       5.17 &      14.55 \\

  high &  529.84 &     326.53 &       2.30 &       7.06 &      93.77 &      31.51 &       0.21 &       0.88 \\
\hline
\end{tabular}  
\label{caseStudyTableTripAnalysis}
\end{table}

Table \ref{caseStudyTableTripAnalysis} presents a cost and convenience
analysis for the ODMTS trips and those using personal vehicles. It
also provides the cost and convenience of the other mode, i.e., the
convenience and cost of using a personal vehicle for those using the
ODMTS and vice-versa. As can be seen, the cost of using the ODMTS is
significantly lower, although personal vehicles would decrease the commute
time significantly for low-income riders. Note however that the ODMTS
has also achieved low commuting times. Riders using personal vehicles do so
because the transit times are simply too large for their trips.

\begin{figure}[!t]
\includegraphics[width=\textwidth]{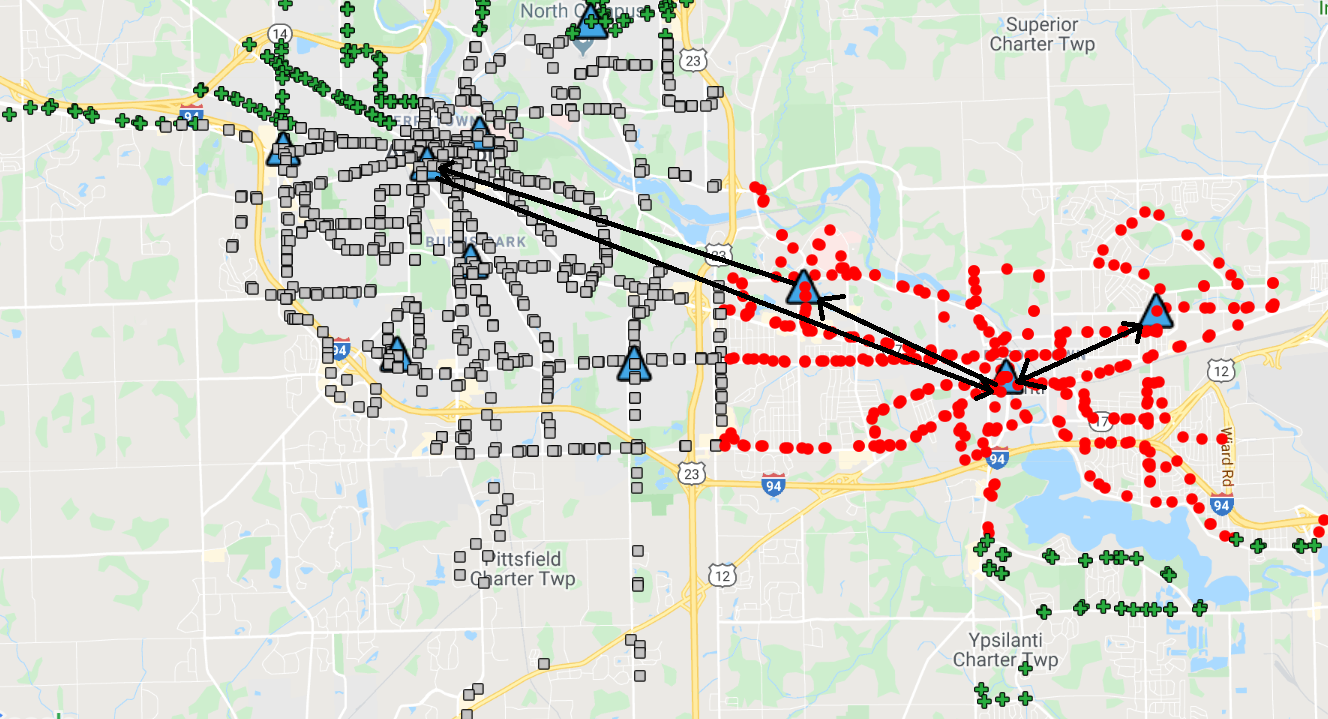}
\caption{Network design for the ODMTS with 10 Hubs with $\Delta = 5$.} \label{networkDesignFigureDelta5miles}
\end{figure} 

The next results examine the effect of the threshold value $\Delta$ on
the rides with on-demand shuttles. Figure
\ref{networkDesignFigureDelta5miles} presents the network design with
$\Delta = 5$ mile. This allows for longer shuttle rides from origin to
destination of each trip compared to the $\Delta = 2$ case in Figure
\ref{networkDesignFigure}. As a result, the network design has fewer
connections between hubs. Although the investment cost for the network
design is lower in this case, the average route cost of the obtained
trips increases and the average time of the trips decreases through
the adoption of more on-demand shuttles. This highlights the trade-off
between the high-frequency buses and on-demand shuttles.

\subsection{The Benefits of the Formulation}
\label{solutionEvaluationSection}

This section compares the novel bilevel formulation with latent demand
(L) with the original formulation that ignores the latent demand
(O). In other words, the original formulation designs the network with
$T \setminus T'$ trips but is evaluated on the complete set $T$ of
trips. The two network designs are then compared in terms of cost and
convenience. To obtain a realistic setting, the share of public
transit is assumed to be 10\% for each income level. The results are
presented in Table \ref{tab:modelComparisonResults}.


\begin{table}
    \centering
    \caption{Comparison of the Proposed Model (L) and the Original Model (O) with $\Delta=5$.}
    \label{tab:modelComparisonResults}
      \begin{tabular}{ccccccc} 
  \hline
  Model & Income &   Adoption & Investment (\$) & ODMTS trips (\$) & Conv. (s) & Cost \& Conv. \\
  \hline
\multicolumn{ 1}{c}{L} &        low &       1.00 & \multicolumn{ 1}{c}{2482.38} & \multicolumn{ 1}{c}{17530.24} & \multicolumn{ 1}{c}{1269263.87} & \multicolumn{ 1}{c}{32505.13} \\

\multicolumn{ 1}{c}{} &     middle &       1.00 & \multicolumn{ 1}{c}{} & \multicolumn{ 1}{c}{} & \multicolumn{ 1}{c}{} & \multicolumn{ 1}{c}{} \\

\multicolumn{ 1}{c}{} &       high &       0.84 & \multicolumn{ 1}{c}{} & \multicolumn{ 1}{c}{} & \multicolumn{ 1}{c}{} & \multicolumn{ 1}{c}{} \\

\multicolumn{ 1}{c}{O} &        low &       1.00 & \multicolumn{ 1}{c}{861.54} & \multicolumn{ 1}{c}{20685.56} & \multicolumn{ 1}{c}{1167457.12} & \multicolumn{ 1}{c}{33006.20} \\

\multicolumn{ 1}{c}{} &     middle &       0.99 & \multicolumn{ 1}{c}{} & \multicolumn{ 1}{c}{} & \multicolumn{ 1}{c}{} & \multicolumn{ 1}{c}{} \\

\multicolumn{ 1}{c}{} &       high &       0.82 & \multicolumn{ 1}{c}{} & \multicolumn{ 1}{c}{} & \multicolumn{ 1}{c}{} & \multicolumn{ 1}{c}{} \\
\hline
\end{tabular}  





\end{table}

The results show that both models have similar results in terms of
mode switching. However, the new formulation has a higher investment
cost and a lower cost for the ODMTS trips compared to the original
formulation. The difference between the models is highlighted in
Figure \ref{networkDesignFigureComparison}, which shows the network
designs under the two approaches: The dashed legs represent the design
under the original model, and the other legs correspond to the design
of the proposed model. This result is intuitive: With more ridership,
the ODMTS should open more legs and further reduces congestion. It
shows that the novel formulation provides a more robust solution that
should reassure transit agencies. As the original formulation opens
fewer legs between hubs, users utilize more on-demand shuttles,
resulting in trips with more convenience but at much higher costs. In
terms of the total investment and trips cost, the results show that
the new and original formulations have total costs of \$20012.62 and
\$21547.10, respectively. As this cost improvement corresponds to a
planning horizon of 4 hours, it scales up to a gain of \$1227585 over
a yearly plan with 200 days over 16 hours. {\em This is significant
  for this case study and highlights why transit agencies are worried
  about the success of ODMTS when they are planned with the existing
  demand only: They will under-invest in bus lines and sustain higher
  shuttle costs.} The formulation proposed in this paper remedies this
limitation: By taking into account the personalized choice models of riders,
the network design invests in high-frequency buses, decreasing the
overall cost while maintaining an attractive level of convenience.
Note also that, the pricing model adopted in this paper keeps the
transit costs low but is also conducive to numerous mode switchings,
since the transit system subsidies half the cost.

\begin{figure}
\includegraphics[width=\textwidth]{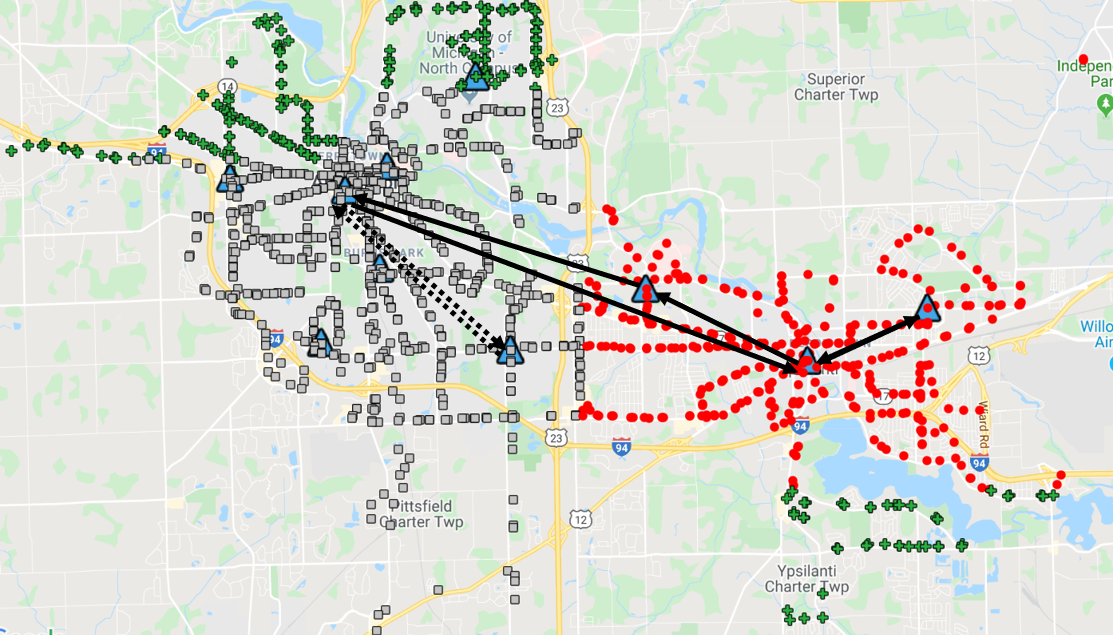}
\caption{Network Designs of the Proposed Model (L) and the Original Model (O).} \label{networkDesignFigureComparison}
\vspace{-4mm}
\end{figure}

It is also important to report the computational performance of the
proposed algorithms. The formulation with latent demand requires 513
seconds to converge in 8 iterations, whereas the original formulation
requires 189 seconds in 8 iterations for the presented case
study. 

\section{Conclusion}
\label{SectionConclusion}

This study presented a bilevel optimization approach for modeling the
ODMTS by integrating rider preferences and considering latent
demand. The transit network designer optimizes the network design
between the hubs for connecting them with high frequency buses,
whereas each rider tries to find the most cost-efficient and
convenient route under a given design through buses and on-demand
shuttles. The paper considered a generic preference model to capture
whether riders switch to the ODMTS based on the obtained route and
their current mode of travel. To solve the resulting optimization
problem, the paper proposed a novel decomposition approach and
developed combinatorial Benders cuts for coupling the network design
decisions with rider preferences. A cut strengthening was also
proposed to exploit the structure of the follower problem and, in
particular, a monotonicity assumption of the choice model.  The
potential of the approach was demonstrated on a case study using a
data set from Ann Arbor, Michigan. The results showed that ignoring
latent demand can lead to significant cost increase (about 7.5\%) for
transit agencies, confirming that these agencies are correct in
worrying about customer adoption. This is the case even for a pricing
model where the transit agency and riders share the shuttle costs.
The new formulation can also be solved in reasonable time.

Current work is devoted to examining the impact of various cost models
and different choice models for riders. Applications of the model to
the city of Atlanta is also contemplated and should reveal some
interesting modeling and computational challenges given the size of
the city.

\vspace{-2mm}

\subsection*{Acknowledgements}

This research is partly supported by NSF Leap HI proposal NSF-1854684.

\bibliographystyle{splncs04}
\bibliography{references}

\end{document}